\documentclass{amsart}
\usepackage{geometry}                
\usepackage{graphicx}
\usepackage{tikz}
\usetikzlibrary{shapes.geometric}
\usetikzlibrary{3d,calc}
\usepackage{amssymb}
\usepackage{epstopdf,amsmath}

\newtheorem{prop}{Proposition}[section]
\newtheorem{lemma}[prop]{Lemma}

\newtheorem{theorem}[prop]{Theorem}
\newtheorem{conj}[prop]{Conjecture}

\newtheorem*{lemma*}{Lemma}
\newtheorem*{theorem*}{Theorem}

\newcommand{\eps}{\epsilon}

\newcommand{\ZZ}{\mathbb{Z}}
\newcommand{\RR}{\mathbb{R}}

\newcommand{\TT}{\mathbb{T}}

\title{Incidence estimates for well spaced tubes}

\author{Larry Guth, Noam Solomon, and Hong Wang}

\begin{document}

\maketitle

\begin {abstract} We prove analogues of the Szemer\'edi-Trotter theorem and other incidence theorems using $\delta$-tubes in place of straight lines, assuming that the $\delta$-tubes are well-spaced in a strong sense.
\end {abstract}

\section{Introduction}

In a series of papers in the late 90s, Tom Wolff explored the connection between incidence geometry and Kakeya-type problems in harmonic analysis.  By adapting techniques from the incidence geometry literature, he was able to prove some striking results in harmonic analysis (cf. \cite{W96}, \cite{W97}, and \cite{W00}).  Incidence geometry is about the intersection patterns of lines, and the Kakeya problem is about the intersection patterns of thin tubes, and so it sounds reasonable that they should be related.  It turns out, however, that it is quite subtle to adapt theorems from the setting of lines to the setting of thin tubes, and a lot remains unknown.  In order to get non-trivial estimates in the setting of tubes, it is necessary to add some assumption about how the tubes are spaced.  There are many possible assumptions, and so there are many different problems to consider.  In this paper, we consider very strong spacing assumptions on the tubes, and under those assumptions we prove nearly sharp incidence estimates.

Our first main result is an analogue of the Szemer\'edi-Trotter theorem from incidence geomery.  We first recall the theorem.  Suppose that $\frak L$ is a set of lines in the plane.  For $r \ge 2$, let $P_r(\frak L)$ denote the $r$-rich points of $\frak L$ - the set of points that lie in at least $r$ lines of $\frak L$.  The Szemer\'edi-Trotter theorem gives sharp bounds for $|P_r(\frak L)|$:

$$ |P_r(\frak L)| \lesssim r^{-3} |\frak L|^2 + r^{-1} |\frak L|. \eqno{(ST)}$$

Now suppose that $\TT$ is a set of $\delta \times 1$ tubes (i.e. rectangles) in $[0,1]^2$.  The set of all $\delta$-balls that intersect at least $r$ tubes of $\TT$ is infinite, and so we define $P_r(\TT)$ to be the set of $\delta$-balls that have centers in the lattice $\delta \ZZ^2$ and intersect at least $r$ tubes of $\TT$.  The bound (ST) does not hold for tubes in this generality.  We begin with a few simple examples to show that some spacing conditions are necessary.  First of all, if all the tubes of $\TT$ are tiny perturbations of a fixed tube, with the size of the perturbation less than $\delta$, then we can get $\delta^{-1}$ $r$-rich $\delta$--balls for $r \sim |\TT|$.  We say that two $\delta$-tubes, $T_1$ and $T_2$, are essentially distinct if

$$ |T_1 \cap T_2| \le (1/2) |T_1|. $$

\noindent From now on we assume that the tubes of $\TT$ are essentially distinct.  But the bound (ST) does not hold for essentially distinct tubes either.  Let $R$ be an $r \delta \times 1$ rectangle.  There are $\sim r^2$ essentially distinct $\delta$-tubes in $R$, and we let $\TT_R$ denote such a set of $\delta$-tubes.  Then $P_r(\TT_R) \sim r \delta^{-1}$, which is often much bigger than (ST).  In the context of the Kakeya problem, one sometimes considers tubes that point in distinct directions. For instance, suppose that $\TT$ consists of $\delta^{-1}$ $\delta$-tubes, all going through the origin, and pointing in $\delta$-separated directions.  In this case, it's not hard to check that $|P_r(\TT)| \sim r^{-2} | \TT|^2$, which is still bigger than the (ST) bound for all $1 \ll r \ll \delta^{-1}$.  

To get an analogue of (ST) for tubes, we need to make a stronger hypothesis about how the tubes are spaced.  We will consider the following hypothesis, which is the strongest spacing condition that we can make.  Fix some $W \ge 1$.  There are $\sim W^2$ essentially distinct $W^{-1} \times 1$ rectangles in $[0,1]^2$.  Then fix some $\delta < 1/W$ and let $\TT$ be a set of $W^2$ $\delta$-tubes, one contained in each of these $W^{-1} \times 1$ rectangles.  Even under this spacing condition (ST) does not always hold.  The reason is that an average $\delta$-ball in $[0,1]^2$ is $r$-rich for $r \sim \delta | \TT|$.  If $ r \le \delta | \TT|$, then for a typical choice of $\TT$, we have $|P_r(\TT)| \sim \delta^{-2}$, which often violates (ST).  Our first theorem says that if $\TT$ is well-spaced in this sense, and if $r$ is bigger than the threshold $\delta | \TT|$, then the (ST) bound holds up to small errors.

\begin{theorem} \label{stwellspaced} Suppose that $1 \le W \le \delta^{-1}$.  Suppose that $\TT$ is a set of $\sim W^{-2}$ $\delta$-tubes in $[0,1]^2$ with $\lesssim 1$ $\delta$-tube of $\TT$ in each $W^{-1} \times 1$ rectangle.
	
$$\textrm{ If } r > \max(\delta^{1-\eps} | \TT|, 1), $$

$$\textrm{then } |P_r(\TT)| \lesssim_\eps \delta^{-\eps} r^{-3} |\TT|^2. $$
	
\end{theorem}

Another variation of our argument estimates the incidences for a set of tubes with many well-spaced tubes in every direction.

\begin{theorem} \label{stw} Let $1 \le W \le \delta^{-1}$.  Divide the circle into arcs $\theta$ of length $\delta$.  For each $\theta$, and each $1 \le j \le W$, let $T_{\theta,j} \subset [0,1]^2$ be a $\delta$-tube.  Suppose that for each $\theta$, and each $W^{-1} \times 1$ rectangle in direction $\theta$, there are $\sim N_1$ tubes $T_{\theta, j}$ in the rectangle.  Let $\TT$ be the set of all the tubes $T_{\theta, j}$.   Then for any $\eps>0$ 
	
	$$\textrm{ if } r \ge C_1(\eps) \delta^{1 - \eps} | \TT|, $$
	
	$$\textrm{ then } |P_r(\TT)| \le C_2(\eps) \delta^{-\eps} W^{-1} r^{-2} | \TT|^2. $$
	
\end{theorem}

\noindent This estimate is also sharp, as we will see below.  This problem came up in conversations with Ciprian Demeter about decoupling theory.  We hope to discuss the connection with decoupling problems in a later paper with him.

We were also able to push our method to three dimensions.  In \cite{GK15}, the first author and Nets Katz proved an incidence estimate for lines in $\RR^3$, which says that if $\frak L$ is a set of lines in $\RR^3$ with at most $|\frak L|^{1/2}$ lines in any plane or degree 2 algebraic surface, then

\begin{equation} \label{incr3} |P_r(\frak L)| \lesssim r^{-2} |\frak L|^{3/2} + r^{-1} |\frak L|. \end{equation}

\noindent We prove an analogue of this estimate for well-separated tubes in three dimensions.

\begin{theorem} \label{gkwellspaced}  Suppose that $1 \le W \le \delta^{-1}$.  Suppose that $\TT$ is a set of $\sim W^{-4}$ $\delta$-tubes in $[0,1]^3$ with $\lesssim 1$ $\delta$-tube of $\TT$ in any tube of radius $W^{-1}$ and length 1.

$$\textrm{ If } r > \max ( \delta^{2-\eps}  | \TT|, 1), $$

$$\textrm{then } |P_r(\TT)| \lesssim_\eps \delta^{-\eps} r^{-2} |\TT|^{3/2}. $$

\end{theorem}

This theorem gives a very special case of the Kakeya conjecture in $\RR^3$.  The Kakeya maximal function conjecture in $\RR^3$ says that if $\TT$ is a set of $\delta^{-2}$ $\delta$-tubes pointing in $\delta$-separated directions, then $|P_r(\TT)| \lesssim_\eps \delta^{-\eps} r^{-3/2} | \TT|^{3/2}$.  Our bound is stronger than this one, but it only applies if the tubes of $\TT$ obey our very strong spacing condition.

The incidence estimate for lines in $\RR^3$ in \cite{GK15} was motivated by the Erd{\H o}s distinct distance problem in the plane.  The problem asks for the minimal number of distinct distances determined by $N$ points in the plane.  In \cite{ES}, Elekes and Sharir proposed an interesting approach to the distinct distance problem which connects it to incidences between points and lines in three dimensions.  Combining their approach with the bound (\ref{incr3}), the paper \cite{GK15} proved that $N$ points in the plane determine $\gtrsim N / \log N$ distinct distances, which is sharp up to logarithmic factors.  Using the Elekes-Sharir framework, Theorem \ref{gkwellspaced} implies a similar distance estimate for well-spaced $\delta$-balls or points. 

\begin{theorem} \label{distdistwellspaced} If $E$ is a set of $N$ points in $[0,1]^2$ with $\lesssim 1$ point in each $N^{-1/2}$-ball, then $\Delta(E)$ contains $\gtrsim_\eps N^{1-\eps}$ distances which are pairwise separated by $N^{-1}$.
\end{theorem}

This theorem is relevant to the Falconer problem which is a kind of continuous analogue of the Erd{\H o}s distinct distance problem.  Falconer asked for the smallest Hausdorff dimension of a compact set $E \subset [0,1]^2$ which guarantees that $\Delta(E)$ has positive measure.  In \cite{Falc86}, Falconer proved that $\dim_H (E) > 3/2$ suffices, and he conjectured that $\dim_H (E) > 1$ suffices.  In \cite{Mat87}, Mattila proposed a Fourier analytic approach to the problem which connects it to restriction theory.  Using that connection, Wolff \cite{W99} proved that $\dim_H (E) > 4/3$ suffices.  Recently, using decoupling, the paper \cite{GIOW} proved that $\dim_H (E) > 5/4$ suffices.  Falconer's conjecture is closely related to the following conjecture about finite sets of balls.

\begin{conj} \label{falcdisc} Suppose that $\alpha > 1$.  Suppose that $E$ is a set of $\delta^{-\alpha}$ $\delta$-balls in $[0,1]^2$, and that any ball of radius $S \delta$ contains $\lesssim \delta^{-\eps} S^\alpha$ balls of $E$.  Then the number of $\delta$-intervals needed to cover $\Delta(E)$ is $\gtrsim_\eps \delta^{-1}$.  
\end{conj}

Theorem \ref{distdistwellspaced} proves this conjecture up a factor of $\delta^\eps$ for sets $E$ that are as widely spaced as possible.  In the other direction, there has been some remarkable work by Orponen \cite{O17} and Keleti-Shmerkin \cite{KS18} on the case when $E$ is tightly spaced.  We say that $E$ is an Ahlfors-David regular set of $\delta$-balls if, for each ball of $E$, the concentric $S \delta$ ball contains $\approx S^\alpha$ balls of $E$.  Orponen's paper \cite{O17} implies that this conjecture holds up to a factor of $\delta^\eps$ for Ahlfors-David regular sets.  

Let us now describe the sharp examples for Theorem \ref{stw}, because these examples indicate an important structure that plays a role in the proofs.  
We pick $W$ balls of side length $A \delta$, for a parameter $A$ to be determined later, with centers evenly spaced along the line segment from $(0,0)$ to $(1,0)$.  For most $\theta$, we can arrange that one tube $T_{\theta, j}$ passes through each ball. So $\sim \delta^{-1}$ tubes pass through each ball.  We call these balls heavy balls.  On average, a point in one of these special balls lies in $A^{-1} \delta^{-1}$ tubes of $\TT$, and by perturbing the tubes by random translations of size $A \delta$, we can assume that most points of most heavy balls lie in $\sim A^{-1} \delta^{-1}$ tubes of $\TT$.  Now we choose $A$ so that $r = A^{-1} \delta^{-1}$.  We compute 

$$ |P_r(\TT)| \gtrsim W A^2 = W r^{-2} \delta^{-2} = W^{-1} r^{-2} | \TT |^2. $$

These heavy balls play an important role in the proof.  One key tool in our proof is a Fourier analysis argument which shows that if there are too many $r$-rich $\delta$-balls, then they have to be organized into larger heavy balls like in this example.  This Fourier analysis argument is based on arguments in the literature on projection theory, especially the recent paper by Orponen \cite{O17b}.  We combine this heavy ball lemma with the idea of partitioning, which comes from the incidence geometry literature.  In \cite{CEGSW}, Clarkson, Edelsbrunner, Guibas, Sharir and Welzl used the idea of partitioning to give a new proof of the Szemer\'edi-Trotter theorem and prove new theorems in incidence geometry, and Wolff in turn built on this partitioning idea in the papers mentioned above.

\section{Finding heavy balls}

\begin{prop} \label{heavyballs} Suppose that $P$ is a set of unit balls in $[0,D]^n$ and $\TT$ is a set of tubes of length $D$ and radius 1 in $[0,D]^n$.  Suppose that each ball of $P$ lies in $\approx E$ tubes of $\TT$.  Let $S$ be a scale.  Then either

\vskip5pt
	
{\bf Thin case.}  $  |P| \lessapprox S^n E^{-2} |\TT| D^{n-1}$, or

\vskip5pt

{\bf Thick case.} There is a set of disjoint $S$-balls $Q_j$ so that

\begin{enumerate}
	\item $ \cup_j Q_j$ contains a fraction $\gtrapprox_n 1$ of the cubes of $P$.
		
	\item Each $Q_j$ intersects $\gtrapprox_n S^{n-1} E$ tubes of $\TT$.
	
\end{enumerate}

\end{prop}

Before giving the proof, let us discuss the numerology.  When we apply the Proposition, $S$ will be small, and so the $S$ factor in the thin case will be negligible.  For the thin case, let us focus on dimension $n=2$.  In this case, if the directions of the tubes $\TT$ are evenly spaced, we would get $|P| \lessapprox E^{-2} | \TT|^2$ (see Lemma \ref{kakmax} below).  If $|\TT|$ is much bigger than $D$, then the bound in the thin case represents a savings.  Now we turn to the thick case.  Without loss of generality, we can think of $P$ as $P_E(\TT)$.  In the thick case, a typical unit cube in one of the $S$-cubes $Q_j$ lies in $\sim E$ tubes of $\TT$, and so morally all the unit cubes of each $Q_j$ lie in $P_E(\TT)$.  This structure matches the heavy balls we saw in the sharp example for Theorem \ref{stw} in the introduction.  However, since we use small $S$, applying this Proposition does not immediately find the whole heavy ball in the example -- only a smaller heavy sub-ball.  In the full proof of Theorem \ref{stw}, we will use this lemma many times in an iteration scheme.

\begin{proof} The proof of Proposition \ref{heavyballs} is based on Fourier analysis.  For each unit cube $q$ of $P$, we let $\psi_q$ be a smooth bump approximating $\chi_q$.  Let $f = \sum_{q \in P} \psi_q$.  For each tube  $T$ in $\TT$, let $\psi_T$ be a smooth bump approximating $\chi_T$.  Let $g = \sum_{T\in \TT} \psi_T$.  If $q$ intersects $T$, then $\int \psi_q \psi_T\gtrsim 1$, and so

$$ I(P,\TT) \lesssim \int f g. $$

We apply Plancherel: $ \int f g = \int \hat f \bar{\hat g}$.  Next we decompose Fourier space into high-frequency and low-frequency pieces.  We let $\rho$ be a frequency cutoff, which is slightly larger than $S^{-1}$: $\rho$ is $D^\eps S^{-1}$ for a tiny $\eps > 0$.  We let $\eta$ be a smooth bump which is equal to 1 on the ball of radius $\rho$, and which is supported in the ball of radius $2 \rho$.  

$$ I(P,L) \lesssim \int \eta \hat f \bar {\hat g} + \int (1 - \eta) \hat f \bar {\hat g}. $$

If the high frequency piece dominates, we will show that the conclusion of thin case holds, and if the low frequency piece dominates, then we will show that the conclusion of the thick case holds.

\vskip10pt

{\bf The high frequency case.} If the high-frequency term dominates, then we have

$$ I(P,L) \lesssim \int (1 - \eta)  \hat f \bar {\hat g} \le \left(\int (1-\eta) |\hat f|^2\right)^{1/2}  \left(\int (1-\eta)  |\hat g|^2\right)^{1/2}. $$

We bound the factor involving $f$ by $\| \hat f \|_{L^2} = \| f \|_{L^2} \sim |P|^{1/2}$.

To bound the factor involving $g$, we take advantage of the support of the Fourier transform of $\psi_T$.  Cover the unit sphere $S^{n-1}$ by $1/D$-caps $\theta$.  Let $\TT_\theta$ be the set of $T \in \TT$ in direction $\theta$, and let $g_\theta = \sum_{T \in \TT_\theta} \psi_T$.  If $T$ is a $1 \times D$ tube in direction $\theta$, then $\hat \psi_T$ is rapidly decaying outside of $\theta^*$.  Here $\theta^*$ is a $D^{-1} \times 1 \times ... \times 1$ slab through the origin perpendicular to $\theta$.  Now we consider the integral

\begin{equation} \label{expandtheta}  \int (1 - \eta) | \hat g |^2 = \int (1 - \eta(\omega)) \left| \sum_\theta \hat g_\theta(\omega) \right|^2 d \omega. \end{equation}

\noindent If $1 - \eta(\omega) \not= 0$, then $|\omega| \ge \rho$.  In that case, $\omega$ belongs to $\theta^*$ for $\lesssim \rho^{-n} D^{n-2}$ different $\theta$.  (For comparison, note that the total number of $\theta$ is $D^{n-1}$.)  Applying Cauchy-Schwarz, we see that for any $N$, 

$$  \left| \sum_\theta \hat g_\theta(\omega) \right|^2 \lesssim \rho^{-n} D^{n-2} \sum_\theta | \hat g(\omega|)|^2 + C_N D^{-N}. $$

The term $C_N D^{-N}$ accounts for the rapidly decaying tails of the functions $\hat g_\theta$.  This term is negligible, and we ignore it in the sequel.  Plugging our bound into (\ref{expandtheta}), we see that

$$ \int(1 - \eta) | \hat g|^2 \lesssim \rho^{-n} D^{n-2} \sum_\theta \int |\hat g_\theta|^2 = \rho^{-n} D^{n-2} \sum_\theta \int |g_\theta|^2. $$

Now for each $\theta$, the tubes $T \in \TT_\theta$ are disjoint, and so

$$  \rho^{-n} D^{n-2} \sum_\theta \int |g_\theta|^2 = \rho^{-n} D^{n-2} \sum_{T \in \TT} \int |\psi_T|^2 \sim \rho^{-n} D^{n-1} |\TT|. $$

Combining what we've done so far, we see that in the high-frequency case

$$ I(P,L) \lesssim \rho^{-n/2} D^{\frac{n-1}{2}} |P|^{1/2} |\TT|^{1/2}. $$

On the other hand, we know that

$$ I(P,L) \approx E |P| . $$

Rearranging, we get

$$ |P| \lesssim  \rho^{-n} E^{-2} D^{n-1} | \TT| \lesssim S^{-n} E^{-2} D^{n-1} | \TT|. $$

\vskip10pt

{\bf The low frequency case.} If the low frequency case dominates, then we have 

$$ I(P,\TT) \lesssim \int \eta \hat f \bar{\hat g} = \int f (g* \eta^{\vee}) = \sum_{q \in P} \sum_{T \in \TT} \int \psi_q (\psi_T* \eta^\vee). $$

Now $\psi_T * \eta^\vee$ is rapidly decaying outside of the $\rho^{-1} \times D$ tube around $T$, and $|\psi_T * \eta^\vee| \lesssim \rho^{n-1}$.  We write $N_S(q)$ for the $S$-neighborhood of $q$, which is essentially a ball of radius $S$.  Since $S = D^\eps \rho^{-1}$, $\psi_T* \eta^\vee$ is negligible outside of the $S \times D$ tube around $T$.  Therefore, 

$$ \sum_{T \in \TT} \int \psi_q (\psi_T* \eta^\vee) \lesssim \rho^{-(n-1)} \# \{ T \in \TT | T \cap N_{S}(q) \not= \emptyset \} \lessapprox S^{-(n-1)} \# \{ T \in \TT | T \cap N_{S}(q) \not= \emptyset \}. $$

Now $I(P,\TT) \approx E |P|$, so 

$$ E |P| \lessapprox \sum_{q \in P}  S^{-(n-1)} \# \{ T \in \TT| T \cap N_{S}(q) \not= \emptyset \}. $$

So for a fraction $\gtrapprox 1$ of $q \in P$, 

$$  \# \{ T \in \TT| T \cap N_{S}(q) \not= \emptyset \} \gtrapprox S^{n-1}  E. $$

This is the desired estimate in the thick case.

\end{proof}

\section{Proof of  Theorem \ref{stw}}

We start by proving Theorem \ref{stw} because the argument is slightly less complicated and because the role of the heavy balls is clearest.  We recall the statement.

\begin{theorem*} 
	 Let $1 \le W \le \delta^{-1}$.  Divide the circle into arcs $\theta$ of length $\delta$.  For each $\theta$, and each $1 \le j \le W$, let $T_{\theta,j} \subset [0,1]^2$ be a $\delta$-tube.  Suppose that for each $\theta$, and each $W^{-1} \times 1$ rectangle in direction $\theta$, there are $\sim N_1$ tubes $T_{\theta, j}$ in the rectangle.  Let $\TT$ be the set of all the tubes $T_{\theta, j}$.   Then for any $\eps>0$ 
	
	$$\textrm{ if } r \ge C_1(\eps) \delta^{1 - \eps} | \TT|, $$
	
	$$\textrm{ then } |P_r(\TT)| \le C_2(\eps) \delta^{-\eps} W^{-1} r^{-2} | \TT|^2. $$
\end{theorem*}

We will first prove that the theorem holds when $W=1$.  This essentially follows from the proof of the Kakeya maximal estimate in the plane.

\begin{lemma} \label{kakmax} Suppose that for each $\delta$-arc $\theta$, $\TT_\theta$ is a set of $\sim N_1$ $\delta \times 1$ rectangles in $[0,1]^2$.  Let $\TT= \cup_\theta \TT_\theta$.  Then
	
	$$ |P_r(\TT)| \lessapprox r^{-2} | \TT|^2. $$
\end{lemma}

\begin{proof} Let $g = \sum_{T \in \TT} \chi_T$.  For any arc of the circle, $\tau$, we let $\TT_\tau$ be the set of tubes of $\TT$ with direction in $\tau$, and we let $g_\tau = \sum_{T\in \TT_\tau} \chi_T$.  We will estimate $\int |g|^2$.  If $\tau_1$ and $\tau_2$ are $\alpha$-arcs and the distance between them is $\sim \alpha$, then
	
$$ \int g_{\tau_1} g_{\tau_2} = \sum_{T_1 \in \TT_{\tau_1}, T_2 \in \TT_{\tau_2}} \int \chi_{T_1} \chi_{T_2} \sim | \TT_{\tau_1} | |\TT_{\tau_2}| \alpha^{-1} \delta^2. $$

Since the directions of tubes are evenly distributed on the circle, $|\TT_{\tau_i}| \sim \alpha | \TT|$, and so

$$ \int g_{\tau_1} g_{\tau_2} \lesssim \alpha | \TT|^2 \delta^2. $$

Now we expand 

$$ \int g^2 = \sum_{s=1}^{\log \delta^{-1}} \sum_{\tau_1, \tau_2 2^{-s} \textrm{ arcs of } S^1, \sim 2^{-s} \textrm{ separated }} \int g_{\tau_1} g_{\tau_2} \lesssim \sum_{s=1}^{\log \delta^{-1}} 2^s 2^{-s} | \TT|^2 \delta^2 \lessapprox |\TT|^2 \delta^2. $$

On the other hand, $\int g^2 \ge |P_r(\TT)| r^2 \delta^2$.  \end{proof}

\begin{proof}[Proof of Theorem \ref{stw}]
The proof is by induction, and there are two base cases.  The first base case is when $W \le \delta^{-\eps/10}$, which follows from Lemma \ref{kakmax}.  

The second base case is when $W$ is almost as big as $\delta^{-1}$.  If $W = \delta^{-1}$, then $\TT$ must consist of essentially all of the $\delta^{-2}$ distinct $\delta$-tubes in $[0,1]^2$.  Formally, the second base case is when $W \ge \delta^{-1 + \eps/2}$.  In this case, we see that $r \ge \delta^{1 - \eps} |\TT| \ge \delta^{1 - \eps} \delta^{-1} W > \delta^{-1 - \eps/2}$.   But the number of distinct $\delta$-tubes through a $\delta$-cube is $\lesssim \delta^{-1}$, and so $P_r(\TT)$ is empty.

Now we begin the inductive argument.  Let $P$ be the set of $\delta$-balls lying in $\sim r$ tubes of $\TT$.  By backwards induction on $r$, we can assume that $|P| \approx |P_r(\TT)|$.

Let $1 \le D \le W$ be a parameter.  In this proof, we will eventually choose $D=W$, but we keep the $D$ notation here to help prepare for another proof in the next section, where we will choose $D$ differently.

We cover the unit square with $D \delta$-balls $Q$.   A tube $T$ intersects $Q$ in a $\delta \times D \delta$ rectangle.  One such rectangles could lie in many tubes $T \in \TT$.  Let $\TT_{Q, M}$ be the set of $\delta \times D \delta$ rectangles in $Q$ which lie in $\approx M$ tubes of $T$.  We choose $M$ to preserve most of the incidences.  More precisely, we can choose $M$ so that 

\begin{equation} \label{goodM} \sum_Q  M I(P \cap Q, \TT_{Q, M}) \gtrapprox I(P, \TT). \end{equation}

\noindent Once we fix $M$, we abbreviate $\TT_{Q,M}$ to $\TT_Q$.  Let $P_{Q, E}$ be the set of $\delta$-cubes of $P \cap Q$ that lie in $\sim E$ tubes of $\TT_Q$.  We choose $E$ so that

\begin{equation} \label{goodE} \sum_Q  M I(P_{Q,E}, \TT_{Q}) \gtrapprox I(P, \TT). \end{equation}

\noindent Once we fix $E$, we abbreviate $P_Q = P_{Q, E}$.  Because each $q \in P$ lies in $\sim r$ tubes of $\TT$, (\ref{goodE}) implies that

$$ |P| \lessapprox \sum_Q |P_Q|. $$

Also, the left hand side of (\ref{goodE}) is $\approx M E \sum_Q |P_Q| \le ME |P|$, and the right-hand side is $\sim r |P|$, and so

$$ M E \gtrapprox r. $$

Next we apply Proposition \ref{heavyballs} to bound each $|P_Q|$.  We set the scale $S$ to be $D^{\eps/10}$.  For each $Q$, we will be in either the thin case or the thick case.

\subsection{Thin case} 

\begin{equation} \label{thincase} |P| \lessapprox \sum_{Q \textrm{ thin}} |P_Q| \lessapprox D^{\eps/10} E^{-2} D \sum_Q | \TT_Q|. \end{equation}

Next we prepare to estimate $\sum_Q | \TT_Q|$.  We cover the circle $S^1$ with arcs $\tau$ of length $D^{-1}$.  For each arc $\tau$, cover the unit disk with parallel $D^{-1} \times 1$ rectangles $R_{\tau, j}$ in the direction $\tau$.  There are $D$ rectangles $R_{\tau, j}$ for each direction $\tau$ and there are $D$ arcs $\tau$, and so the total number of rectangles is $\sim D^2$.  Each $\delta \times 1$ rectangle $T$ lies in $2 R_{\tau, j}$ for some $\tau, j$.  Assign each $T$ to one of the rectangles $R$ containing it. Let $\TT_R$ be the set of $T \in \TT$ assigned to the rectangle $R$.  For each $R$, $| \TT_R| \sim D^{-2} | \TT|$. 

By construction, each $\delta \times D \delta$ tube $S \in \TT_Q$ lies in $\approx M$ tubes $T \in \TT$.  All these tubes $T \in \TT$ make an angle of at most $D^{-1}$ with each other, and so they all belong to $\TT_R$ for a single rectangle $R$.  We let $S_M(\TT_R)$ be the set of $\delta \times D \delta$ rectangles which lie in $\sim M$ tubes of $\TT_R$.  We have

\begin{equation} \label{sumq} \sum_Q | \TT_Q| \lesssim \sum_R |S_M(\TT_R)|. \end{equation}

To study $S_M(\TT_R)$, we rescale $R$.  If we magnify $R$ by a factor of $D$ in the short direction, we get a unit square, and the tubes of $\TT_R$ are transformed to rectangles in this unit square with dimensions $D \delta \times 1$.  Let $\TT_R'$ denote this set of $D \delta \times 1$ rectanges in the unit square.  There is a one-to-one correspondence between $S_M(\TT_R)$ and $P_M(\TT_R')$.  Since $D \le W$, the set $\TT_R'$ obeys the hypotheses of Theorem \ref{stw} with

$$ \tilde \delta = D \delta; \tilde W = D^{-1} W. $$

At this point, we choose $D = W$.  Now $\TT_R'$ consists of $N_1$ $\tilde \delta \times 1$ tube in each direction.  We estimate $|P_M(\TT'_R)|$ using Lemma \ref{kakmax}:

$$|P_M (\TT_R')| \lessapprox M^{-2} |\TT_R'|^2 \lessapprox M^{-2} W^{-4} |\TT|^2. $$

Plugging into (\ref{sumq}) gives

$$ \sum_Q |\TT_Q'| \le \sum_R |P_M(\TT_R')| \lessapprox M^{-2} W^{-2} | \TT|^2.$$

Finally, we plug this into (\ref{thincase}) and note that $D=W$ to get

$$ |P_r(\TT)| \lesssim |P| \lessapprox W^{\eps/10} E^{-2} M^{-2} W^{-1} | \TT|^2. $$

Since $M E \gtrapprox r$ and $W \le \delta^{-1}$, we get the deisred bound.

\subsection{Thick case}  Otherwise we are in the thick case, meaning that

\begin{equation} \label{thickcase} |P| \lessapprox \sum_{Q \textrm{ thick}} |P_Q|. \end{equation}

Recall that if $Q$ is thick, there is a set of disjoint $S \delta$-balls $Q_j \subset Q$ so that

\begin{enumerate}
	\item $\cup_j Q_j$ contains a fraction $\gtrapprox 1$ of the cubes of $P_Q$. 
	
	\item Each $Q_j$ intersects $\gtrapprox S E$ tubes of $\TT_Q$.
	
\end{enumerate}

We let $\tilde P$ be the union of all the $Q_j$ from all the thick $Q$.  Each $\delta \times D \delta$ tube of $\TT_Q$ lies in $M$ tubes of $\TT$, and so each $Q_j$ intersects $\gtrapprox S M E \gtrapprox S r$ tubes of $\TT$.  Therefore,

\begin{enumerate}
	\item $\cup_{Q_j \in \tilde P} Q_j$ contains a fraction $\gtrapprox 1$ of the cubes of $P$.
	
	\item Each $Q_j$ in $\tilde P$ intersects $\gtrapprox S r$ tubes of $\TT$.  
\end{enumerate}

Recall that we chose $S = D^{ \eps/10} \le \delta^{-\eps/10}$.  If $S \delta > W^{-1}$, then it follows that  $W \ge \delta^{-1 + \eps/2}$ and we are in the second base case of the induction.  

Otherwise, $S \delta < W^{-1}$.  In this case, we thicken our $\delta$-tubes to $S \delta$-tubes.  For a given $N$, define $\tilde \TT_N$ to be the set of $S \delta \times 1$ tubes containing $\approx N$ tubes of $\TT$.  By pigeonholing, we can choose $N$ so that the tubes of $\tilde \TT_N$ contain a fraction $\gtrapprox 1$ of the incidences between $\TT$ and $\tilde P$.  We fix $N$, and we define $\tilde \TT = \tilde \TT_N$.  We have $| \tilde \TT| \lesssim N^{-1} | \TT|$.  A typical $S \delta$-ball of $\tilde P$ is $\tilde r \gtrapprox N^{-1} S r$-rich for $\tilde \TT$.  Since each such $S \delta$-ball contains only $S^2$ $\delta$-balls,

\begin{equation} \label{PrT} |P_r(\TT)| \lessapprox S^2 |P_{ \tilde r} (\tilde \TT)|, \textrm{ where } \tilde r \gtrapprox N^{-1} S r.  \end{equation}

Now we will apply induction on $\delta$ to bound $P_{\tilde r} (\tilde \TT)$.  The set $\tilde \TT$ essentially obeys the hypotheses of Theorem \ref{stw} with $\tilde \delta = S \delta$ and $\tilde W = W$.   Since $S \delta < W^{-1}$, $1 \le W \le \tilde \delta^{-1}$, which checks the first hypothesis.  The number of tubes in $\tilde \TT$ is at most $N^{-1} | \TT|$, and the number of tubes of $\tilde \TT$ contained in a $W^{-1} \times 1$ rectangle is at most $W^{-2} N^{-1} | \TT|$.  By adding tubes to $\tilde \TT$, we can arrange that $| \tilde \TT| \sim N^{-1} | \TT|$, and the number of tubes of $\tilde \TT$ contained in a $W^{-1} \times 1$ rectangle is $\sim W^{-2} N^{-1} | \TT| = W^{-2}  | \tilde \TT|$.   Finally we have to check that $\tilde r$ is big enough:

$$ \tilde r \gtrapprox N^{-1} S r \ge N^{-1} S C_1(\eps) \delta^{1-\eps} | \TT| \sim S^\eps C_1(\eps) (S \delta)^{1 - \eps} | \tilde \TT|. $$

\noindent Now we can inductively apply Theorem \ref{stw} at scale $S \delta$ to get

$$ |P_{ \tilde r} (\tilde \TT)| \lessapprox C_2(\eps) (S \delta)^{-\eps} W^{-1} (N^{-1} S r)^{-2} | \tilde \TT|^2 \lessapprox C_2(\eps) (S \delta)^{-\eps} S^{-2} W^{-1} r^{-2} | \TT|^2. $$

Plugging this into Equation (\ref{PrT}), we get

$$ |P_r(\TT)| \lessapprox C_2(\eps) (S \delta)^{-\eps} W^{-1} r^{-2} | \TT|^2. $$

We claim that this gives the desired bound for $|P_r(\TT)|$ and closes the induction.  To check this, we have to see that $S$ is big enough so that $S^{-\eps}$ dominates the implicit factor in the $\lessapprox$.  This indeed happens, because $S = D^{\eps/10}$ and we chose $D = W \ge \delta^{-\eps/10}$, and so $S$ is at least a small negative power of $\delta$.  
\end{proof}

Remarks. The step of writing $\TT$ as a union of $\TT_R$ in the thin case is the partitioning idea from \cite{CEGSW}.  Indeed we have partitioned the set of tubes $\TT$ into $D^2$ equal sets $\TT_R$, and a ball $q \in P$ can belong to at most $D$ of them.

\section{Proof of Theorem \ref{stwellspaced} and Theorem \ref{gkwellspaced}}

In this section, we prove Theorem \ref{stwellspaced} and Theorem \ref{gkwellspaced}.  It uses the ideas from the last proof, but there is also a new idea needed especially for the 3-dimensional result, Theorem \ref{gkwellspaced}.  In the proof of Theorem \ref{stw}, we used the Kakeya maximal estimate in two dimensions as a base case.  The Kakeya maximal conjecture in higher dimensions is a deep open problem.  Because of this, we cannot prove the direct 3-dimensional generalization of Theorem \ref{stw}, which would concern a set of $W^{-2} \delta^{-2}$ $\delta$-tubes $\TT$ in $[0,1]^3$ with $W^2$ well-spaced tubes in each direction.  But the spacing condition in Theorem \ref{gkwellspaced} is different and more useful.

The following Theorem combines Theorem \ref{stwellspaced} and Theorem \ref{gkwellspaced}.

\begin{theorem}\label{main}  Let $1 \le W \le \delta^{-1}$.  
	Let $\mathbb{T}$ be a collection of tubes of radius $\delta$, length $1$ in $B^n(0,2)$, for $n=2, 3$. If for every distinct $1/W$--tube, there is one $l\in \mathbb{T}$ (hence $|\mathbb{T}|\approx W^{2(n-1)}$),  then  for $r>\max( \delta^{n-1-\epsilon/4} |\mathbb{T}|, 1)$ the number of $r$--rich $\delta$--balls is bounded by
	\begin{equation}
	|P_r(\mathbb{T})|\lesssim \delta^{-\epsilon} \frac{|\mathbb{T}|^{\frac{n}{n-1}}}{r^\frac{n+1}{n-1}}.
	\end{equation}
\end{theorem}

\begin{proof}
	We will prove the theorem by induction. There are two base cases for our induction.
	
	The first base case is when $W=O(1)$ and $r>1$. Since the tubes through one point are $1/W\approx O(1)$--separated, we have $|P_r(\mathbb{T})|\lesssim |\mathbb{T}|^2 =O(1)$. 
	
	The second base case is when $W \approx \delta^{-1}$. In this case, $\TT$ is essentially the set of all distinct $\delta$-tubes in $B^n$.  More precisely, the second base case is when $W > \delta^{-1 + \eps/10n}$.  In this case $r\gtrsim \delta^{-\epsilon/4} W^{n-1} > \delta^{-(n-1) - \eps/10}$.  But a $\delta$-cube can lie in at most $\lesssim \delta^{-(n-1)}$ distinct $\delta$-tubes, and so $|P_r(\mathbb{T})|=0$. 

In the inductive argument, we distinguish between the case when $r$ is small and the case when $r$ is large.  The main new ingredient in this proof is a way to handle the small $r$ case.

\subsection{The case $r < \delta^{-\eps^3}$}
	
After dyadic pigeonholing, we can assume that the maximal angle between tubes passing through a typical cube of $P_r(\TT)$ is $\alpha > \delta$.  If $\alpha$ is much smaller than 1, then we cover the sphere $\mathbb{S}^{n-1}$ by caps $\tau$ of radius $\alpha$.  For each $\tau$, 
	we divide the unit cube into cells $\Box_{\tau}$, where each $\Box_{\tau}$ is a thick tube of length $1$ and radius $\alpha$, pointing in the direction defined by $\tau$. The number of $\Box_{\tau}$ is $\alpha^{-2(n-1)}$. 
	Let $\mathbb{T}_{\Box_{\tau}}$ denote the collection of tubes inside $\Box_{\tau}$. Then $|\mathbb{T}_{\Box_{\tau}}|\approx (\alpha W)^{2(n-1)}$. 
	
	We rescale $\Box_{\tau}$ to the unit cube.  The tubes $\TT_{\Box_\tau}$ become $\tilde \TT$.   These tubes $\tilde \TT$ obey the hypotheses of Theorem \ref{main} with 
	$\widetilde{W}=\alpha W$ and  $\widetilde{\delta}=\delta/\alpha$.  The maximal angle between tubes in a typical 2-rich point of $\tilde \TT$ is $\gtrsim 1$.  
	
By induction on $\delta$, 
	
	\begin{equation} \label{transcase} P_2(\tilde \TT) \lesssim \tilde \delta^{-\eps} | \tilde \TT|^{\frac{n}{n-1}} \end{equation}
	 
	 \noindent Each 2-rich $\tilde \delta$-cube of $\tilde \TT$ corresponds to a tube of radius $\delta$ and length $\alpha^{-1} \delta$ in $P_2(\TT_{\Box_\tau})$.  Therefore,
	
	$$ |P_2( \TT)| \lesssim \alpha^{-2(n-1)} \alpha^{-1} |P_2(\tilde \TT)| \lesssim \tilde \delta^{-\eps} \alpha^{-2n+1} | \tilde \TT|^{\frac{n}{n-1}} \le \delta^{-\eps} \alpha | \TT|^{\frac{n}{n-1}}. $$

\noindent Since $r < \delta^{-\eps^3}$, the induction closes if $\alpha < \delta^{10 \eps^3}$.

Next suppose that $\alpha \ge \delta^{10 \eps^3}$.  Now we have two subcases.  If $W > \delta^{-1/2 + \eps/10n}$, then the trivial bound $|P_r(\TT)| \le \delta^{-n}$ suffices.  Indeed the desired upper bound for $|P_r(\TT)|$ is 

$$ \delta^{-\eps} r^{-\frac{n+1}{n-1}} | \TT|^\frac{n}{n-1} \gtrsim \delta^{-\eps} \delta^{10 \eps^3} W^{2n} \ge \delta^{- n - \eps/2}. $$

So now suppose that $W \le \delta^{-1/2 + \eps/10n}$, which implies that $\delta < W^{-2}$, and indeed that 

$$ \delta \le \delta^{\frac{\eps}{5n}} W^{-2}. $$

\noindent 

Let $\rho = W^{-2}$, and let $\tilde \TT$ be the set of tubes formed by thickening each $\delta$-tube of $\TT$ to a $\rho$-tube.  (Since each $1/W$-tube contains one tube of $\TT$, each $\rho$-tube of $\tilde \TT$ also contains only one tube of $\TT$ and $| \tilde \TT| = | \TT|$.)  Cover $B^n(0,2)$ with $\rho$-cubes, and let $Q_{X,M}$ be the set of $\rho$-cubes containing $\sim X$ cubes of $P_r(\TT)$ and intersecting $\sim M$ tubes of $\TT$.  We can choose $X, M$ so that $\cup_{Q \in Q_{X,M}} Q$ contains a $\gtrapprox 1$ fraction of $P_r(\TT)$.  Because $\alpha \ge \delta^{10 \eps^3}$, we have 

$$X\lesssim \delta^{-O(\eps^3)} M^2.$$ 

The number of $M$--rich $\rho$--ball is, by induction on scale, 
	$$P_{M} (\tilde \TT) \lesssim \rho^{-\epsilon} \frac{|\tilde \TT|^\frac{n}{n-1}}{M^{\frac{n+1}{n-1}}}.$$
Therefore,

$$ |P_r(\TT)| \lessapprox \delta^{-O(\eps^3)} M^2 |P_M (\tilde \TT)| \lesssim \rho^{-\eps} M^{2 - \frac{n+1}{n-1}} |\tilde \TT|^{\frac{n}{n-1}}. $$

If $n = 2$ or 3, then $\frac{n+1}{n-1} \ge 2$, and the power of $M$ is $\le 0$.  To close the induction, we need to check that $\delta^{-O(\eps^3)} \rho^{-\eps} \le \delta^{-\eps} r^3$.  Since $r = \delta^{- O(\eps^3)}$ it suffices to check that $ \rho^{-\eps} \le \delta^{-\eps} \delta^{O(\eps^3)}$.  But $\rho / \delta \ge \delta^{-\eps/5n}$, and so this is true.  This finishes the induction in the small $r$ case.  

\subsection{The case $r \ge \delta^{-\eps^3}$}

Now we turn to the induction in the large $r$ case.  The rest of the proof is parallel to Theorem~\ref{stw}.  

Let $1 \le D \le W$ be a parameter.  In this proof, we will eventually choose $D$ to be a small power of $\delta$.  

We cover the unit square with $D \delta$-cubes $Q$.   We let $P$ be the set of $-r$ rich cubes in $P_r(\TT)$, and we can assume by induction on $r$ that $|P| \sim |P_r(\TT)|$.  We pigeonhole as in the proof of Theorem \ref{stw}: $\TT_Q$ is a set of tubes in $Q$ of length $D \delta$ and radius $\delta$, which each belong to $\approx M$ tubes of $\TT$, and $P_Q \subset P_{\sim r}(\TT) \cap Q$ is a set of cubes which each belong to $\approx E$ tubes of $\TT_Q$, where

$$ M E \approx r, \textrm{ and }$$

$$ \sum_Q |P_Q| \approx |P|, \textrm{ and }$$

$$ 1 \le E \le D^{n-1}. $$

Next we apply Proposition \ref{heavyballs} to bound each $|P_Q|$.  We set the scale $S$ to be $D^{\eps/10n}$.  For each $Q$, we will be in either the thin case or the thick case.

In the thin case, we have

$$ |P_r(\TT)| \lessapprox \sum_{Q \textrm{ thin}} |P_Q| \lessapprox \sum_Q S^n E^{-2} D^{n-1} | \TT_Q| \sim E^{-2} D^{n-1 + \eps/10}  \sum_Q | \TT_Q|. $$

To estimate $\sum_Q | \TT_Q|$, we cover the sphere $\mathbb{S}^{n-1}$ by caps $\tau$ of radius $1/D$.  For each $\tau$, 
we divide the unit cube into cells $\Box_{\tau}$, where each $\Box_{\tau}$ is a thick tube of length $1$ and radius $1/D$, pointing in the direction defined by $\tau$. The number of cells $\Box_{\tau}$ is $D^{2(n-1)}$. 
A tube $S$ in $\TT_Q$ lies in $M$ different tubes of $\TT$, and they must all lie in the same cell $\Box_\tau$.  

Let $\TT_{\Box_\tau}$ be the set of $T \in \TT$ contained in $\Box_\tau$.  Rescale $\Box_\tau$ to the unit cube, and let $\tilde \TT = \tilde \TT(\Box_\tau)$ be the resulting set of tubes.  For each $\Box_\tau$, $\tilde \TT$ obeys the hypotheses of this theorem with 
	
	\begin{enumerate}
		\item $\widetilde{\delta}=D\delta$. 
		\item $\widetilde{r}=M \approx E^{-1} r $,
		\item $\widetilde{W}=W/D$, 
	\end{enumerate}

We have 

$$ \sum_Q | \TT_Q | \lesssim D^{2(n-1)} | P_M (\tilde \TT)|. $$

We will apply induction to bound $|P_M(\tilde \TT)|$.  Before we can apply the theorem, we have to verify that $M$ is sufficiently large:  $M > \tilde \delta^{n-1 -\eps/4} | \tilde \TT|$ and $M \ge 2$.   We check the first bound on $M$ by calculation:

$$ M \gtrapprox E^{-1} r > E^{-1} \delta^{n-1 - \eps/4} | \TT| \sim E^{-1} D^{n-1 + \eps/4} \tilde \delta^{n-1-\eps/4} | \tilde \TT| > \tilde \delta^{n-1-\eps/4} | \tilde \TT|. $$

\noindent (The last inequality is because $E \le D^{n-1}$.)  To check $M \ge 2$, we recall that $r$ is big, and we choose $D$ small.  Recall that we are in the case $r > \delta^{-\eps^3}$.  We set $D \sim \delta^{-\eps^4}$, and then $M \approx E^{-1} r \ge D^{-(n-1)} r$ and so $M \ge 2$.  We have to deal with the small $r$ case separately because of this step of the argument.

We have now confirmed that $M$ is sufficiently large, and we can apply induction, giving:

$$  \sum_Q | \TT_Q| \lessapprox (D \delta)^{-\eps} D^{2 (n-1)} M^{- \frac{n+1}{n-1}} ( D^{-2(n-1)} | \TT|)^{\frac{n}{n-1}} \sim (D \delta)^{-\eps} D^{-2} M^{-\frac{n+1}{n-1}} | \TT|^{\frac{n}{n-1}}. $$

Hence

$$ |P_r(\TT)| \lessapprox \delta^{-\eps} D^{n-1 - 2 -\eps/2} E^{-2} M^{-\frac{n+1}{n-1}} |\TT|^{\frac{n}{n-1}}. $$

We now check that this closes the induction.  If $n=3$, the right-hand side is 

$$ D^{-\eps/2} \delta^{-\eps} (ME)^{-2} |\TT|^{3/2} \approx D^{-\eps/2} \delta^{-\eps} r^{-2} | \TT|^{3/2}.$$ 

If $n=2$, the right-hand side is

$$ D^{-\eps/2} \delta^{-\eps} D^{-1} E^{-2} M^{-3} |\TT|^2 \le D^{-\eps/2} \delta^{-\eps} E^{-3} M^{-3} |\TT|^2 \approx D^{-\eps/2} \delta^{-\eps} r^{-3} | \TT|^2.$$

In the thick case, we have a set of $S\delta$ cubes $\tilde P$ so that $\tilde P$ covers a fraction $\gtrapprox 1$ of $P$, and each cube of $\tilde P$ intersects $\gtrapprox S^{n-1} r$ tubes of $\TT$.  Let $\tilde \TT$ be the set of tubes formed by thickening each $\delta$ tube of $\TT$ to a $\rho$-tube.  We see that $\tilde P \subset P_{\tilde r}(\tilde \TT)$ for $\tilde r \gtrapprox S^{n-1} r$.  The tubes $\tilde \TT$ obey the hypotheses of our theorem with $\tilde \delta = S \delta$, $\tilde W = W$, and $|\tilde \TT| = | \TT|$.  (We just have to check that $W \le S\delta^{-1}$.  But by the second base case, we know that $W \le \delta^{-1+\eps/10n}$ and we chose $S \le D = \delta^{-\eps^4}$, so this holds.)

Therefore,

$$ |P_r(\TT)| \lessapprox S^n | \tilde P | \le S^n | P_{\tilde r} (\tilde \TT)| \lessapprox (S \delta)^{-\eps} S^n (S^{n-1} r)^{-\frac{n+1}{n-1}} | \TT|^{\frac{n}{n-1}} \le S^{-1} \delta^{-\eps} r^{-\frac{n+1}{n-1}} | \TT|^{\frac{n}{n-1}}. $$

This closes the induction in the thick case and finishes the proof. 	
	
\end{proof}

Remark.  The statement of Theorem \ref{main} makes sense for all dimensions $n$.  We don't know any counterexamples, and it seems plausible to us that it is true for all $n$.  In our proof, we used $n \le 3$ in the calculation in several places.  

There is an example in the appendix of \cite{GK15} showing that Theorem~\ref{main} is sharp and leads to the conjectured statement for all dimensions $n$. 
We modify the example slightly to accommodate the $\delta$--tube version. 
Let $G$ be the grid $ (\mathbb{Z}/W)^{n-1}  \cap [0,1]^{n-1}$. 
If $a, b \in G$, we take the lines in $\mathcal{L}$ to be the line from $(a,0)$ to $(b,1)$. 

And we take the tubes in $\mathbb{T}$ to be the $\delta$--neighborhood of line segments from $(a, 0)$ to $(b, 1)$. For any pair of tubes $l_1, l_2\in \mathbb{T}$, either they have distance $1/W$ or their angle is $1/W$--separated. So we have verified the assumption of Theorem~\ref{main} for  the tubes in $\mathbb{T}$. 

The calculations in the appendix of \cite{GK15} showed that $$P_r(\mathcal{L})\gtrsim \frac{|\mathcal{L}|^{\frac{n}{n-1}}}{r^{\frac{n+1}{n-1}}}.$$ 

An $r$--rich point $x$ is in the form 
\begin{equation}\label{r rich}
x=(\frac{q-p}{q} a +\frac{p}{q}b, \frac{p}{q})
\end{equation}
where $q/10\leq p< q$ are co-prime positive integers with $q\sim W r^{-\frac{1}{n-1}}$ and the first $n-1$ coordinates have value in $[1/4, 3/4]$. 

Given $x$ and $x'$ in the form of equation~\ref{r rich}, if $x\neq x'$ then $|x-x'|\gtrsim 1/W.$
The reason is the following: if $x_n\neq x_n'$, then $|x-x'|\geq |x_n-x_n'|\geq W^{-1} r^{\frac{1}{n-1}}$; otherwise $x_n=x_n'$ we have either $a\neq a'$ or $b\neq b'$. Both cases are essentially the same because $p/q\approx (q-p)/q\approx 1$ and any two distinct elements in $G$ are $1/W$--separated. Hence, in both cases $|x-x'|\gtrsim 1/W$. 

Now we have showed that the points in $P_r (\mathcal{L})$ are $1/W$--separated. Since $\delta\leq 1/W$, we can thicken the points in $P_{r}(\mathcal{L})$ and they become disjoint $r$--rich $\delta$--balls in $P_r(\mathbb{T})$.

%
%

\section{An application to the Falconer problem}
In this section, we consider a distinct distances type problem for $\delta$-balls in $\RR^2$, which is related to the Falconer distance problem in $\RR^2$.   As we mentioned in the introduction,  Orponen\cite{O17} and Keleti-Shmerkin\cite{Shm17A}\cite{Shm17B}\cite{KS18} essentially solved the Falconer distance problem for sets that are close to Ahlfors-David regular.  Here we consider the opposite type of set -- Ahlfors-David regular sets of a given dimension are packed as tightly as possible, and we consider here sets that are as spread out as possible.  

If $E$ is a set in the plane, recall that $\Delta(E)$ is the distance set

$$\Delta(E)=\{ |x-y|, x, y \in E\},$$  where $|x-y|$ denote the Euclidean distance between two points $x$ and $y$.

\begin{theorem}\label{tent} Fix $1 < s < 2$.  Suppose that $E$ is a set of $\delta^{-s}$ $\delta$-balls in $[0,1]^2$, with $\lesssim 1$ $\delta$-ball in each ball of radius $\delta^{s/2}$.  Then the number of disjoint $\delta$-intervals contained in $\Delta(E)$ is $\gtrsim_\eps \delta^{-1 + \eps}$ for all $\eps> 0$.
\end{theorem}

We let $\#\Delta(E)$ denote the number of disjoint $\delta$--intervals contained in $\Delta(E)$. 

We can choose  two balls  $B_1$ and $B_2$ of radius $1/10$ and with centers about $1/3$ part such that  each $E_j = E\cap B_j$ contains about $1/100$ of $E$.  It suffices to show that  $\Delta(E_1, E_2)=\{|x-y|, x\in E_1, y\in E_2\}$ contains $\gtrsim \delta^{-1+\epsilon}$ many $\delta$--intervals.

We recall the Elekes-Sharir framework\cite{ES}, which was used in the Erd\H{o}s distinct distance problem\cite{GK15}. 

If $|x_1-y_1|=|x_2-y_2|$ for points $x_1, x_2\in E_1$ and $y_1, y_2\in E_2$, then there exists a unique (orientation-preserving) rigid motion $g$ on the plane  sending $x_1$ to $y_2$ and $y_1$ to $x_2$. 
A rigid motion $g =(c,\theta)$ is uniquely determined by the center $c\in \mathbb{R}^2$ and the rotation angle $\theta$. We could represent $g$ by a point $\rho(g)=(c, \cot \frac{\theta}{2})$ in $\mathbb{R}^3$. Let $g_{xy}$ denote the collection of rigid motions sending a point $x$ to $y$. Then  $\rho(g_{xy})$  a line in $\mathbb{R}^3$:
\begin{equation}\label{ES tube}
l_{xy}=\rho(g_{xy}) = (\frac{x_1+y_1}{2},\frac{x_2+y_2}{2},  0) + t(-\frac{y_2-x_2}{2}, \frac{y_1-x_1}{2}, 1).
\end{equation}
  In particular, the centers of those $g$ lie on the perpendicular bisector of $x$ and $y$. We can also read the coordinates of $x$ and $y$ from the parameterized equation of $l_{x,y}$. If a line $l$ is parametrized by $$l=\{ ( c_1, c_2, 0) +t(k_1, k_2, 1)\},$$
  then there exists $x=(x_1, x_2)$ and $y=(y_1, y_2)$ such that $l_{xy}=l$. To find $x, y$, it suffices to solve the linear equations system:
  $$ x_1+y_1=2c_1, x_2+y_2=2c_2, y_2-x_2=-2k_1, y_1-x_1=2k_2.$$
  Hence, a line in $\mathbb{R}^3$ one-by-one corresponds to a pair of points in $\mathbb{R}^2$. 

Above is the discrete version of the  Elekes-Sharir framework.  In order to treat the $\delta$--thickening variation, we need a few notations. Let $p$ be a $\delta$--ball in $E_1$ and $q$ be a $\delta$--ball in $E_2$. We say that a rigid motion $g$ sends $p$ to $q$ if $g(p)\cap q\neq\emptyset$.
\begin{lemma}\label{tube}
	Assume that  $dist(p,q)\approx 1$, and  $B$ is a ball of radius about $1$ containing $p$ and $q$. Let $g_{p,q, B}$ denote the collection of rigid motions sending $p$ to $q$ with centers in $B$. Then $\rho(g_{p,q, B})$ is approximately a tube $l_{p,q}$ of radius $\delta$, length about 1. 
\end{lemma}
\begin{proof}
	Let $x$ be the center of $p$ and $y$ be the center of $q$, then $l_{p,q, B}$ lies inside $$\{ (\frac{x_1+y_1}{2}+O(\delta), \frac{x_2+y_2}{2}+O(\delta), 0) + t( -\frac{y_2-x_2}{2} +O(\delta), \frac{y_1-x_1}{2}+O(\delta), 1) \}.$$
	
	By equation~\ref{ES tube}, the angle between $l_{xy}$ and the $\{z=0\}$--plane is $\arctan(\frac{|x-y|}{2})$. Since $dist(p,q)\approx 1$, when $x'$ and $y'$ move within $p, q$, the angle $\arctan(\frac{|x'-y'|}{2})$ is about $45$ degrees and moves about $\delta$. So $l_{p, q, B}$ is about a tube of radius $\delta$ and length $O(1)$.   
\end{proof}

If $p_1, p_2$ are two $\delta$--balls in $E_1$ and $q_1, q_2$ are two $\delta$--balls  in $E_2$ such that $$|dist(p_1,q_1)-dist(p_2,q_2)|< \delta,$$ then there exists a rigid motion $g$ sending $p_1$ to $q_2$ and $q_1$ to $p_2$. Moreover, $\rho$  maps the set of such rigid motions to  a $\delta$--ball in $\mathbb{R}^3$. 

\begin{lemma}\label{ball}
	Suppose that $p_j, q_j$ are disjoint $\delta$--balls satisfying:  $dist(p_1, p_2)\leq 1/10$, $dist(q_1, q_2)\leq 1/10$, $dist(p_1, q_1)\geq 1/3$ and 
	$$|dist(p_1, q_1)-dist(p_2,q_2)|<\delta.$$
	Then $\rho(g_{p_1, q_2}\cap g_{q_1, p_2})$ is roughly a ball of radius $\delta$.  (Here we do not distinguish a shape $\Omega$ with a ball of radius $\delta$ if $\Omega$ contains a ball of radius $O(\delta)$ and is contained in a ball of radius $O(\delta)$.)
\end{lemma}
\begin{proof}
	Let $B_1$ be a ball of radius $1/10$ containing $p_1$ and $p_2$, and $B_2$ be another ball of radius $1/10$ containing $q_1$ and $q_2$. Then $dist(B_1, B_2)\geq 1/10$. 
	
	If a rigid motion $g$ sends $p_1$ to $q_2$ and $q_1$ to $p_2$, then $g$  is roughly a reflection between $B_1$ and $B_2$: $g(B_1)\cap B_2\neq \emptyset$ and $g(B_2) \cap B_1 \neq \emptyset$. So the center of $g$ must lie in a  ball $B_3$ of radius $1/5$ with center in the midpoint of $p_1$ and $q_1$. 
	
	Let $B= 20B_3$. By Lemma~\ref{tube}, $\rho$ maps the collection of rigid motions sending $p_1$ to $q_2$ with centers in $B$ to a tube $l_{p_1, q_2, B}$ of radius $\delta$.

	Now we would like to understand how $l_{p_1, q_2, B}$ and $l_{q_1, p_2, B}$ intersect. 
	If $x, y$ are centers of $p_1$ and $q_2$, then $l_{x,y}$ and $l_{y,x}$ intersects transversely because 
$$ |(-\frac{y_2-x_2}{2}, \frac{y_1-x_1}{2}, 1) \times (-\frac{x_2-y_2}{2}, \frac{x_1-y_1}{2}, 1)|\gtrsim 1.$$
Since  $p_2$ is not too far away from $p_1$, and $q_1$ is not too far away from $q_2$, two tubes $l_{p_1, q_2, B}$ and $l_{q_1, p_2, B}$ intersect transversely.

\end{proof}

From the proof of Lemma~\ref{ball}, we can restrict to the rigid motions with center in $[0,1]^2$. And for any pair of $\delta$--balls $p\in E_1$ and $q\in E_2$, the collection of interesting rigid motions sending $p$ to $q$ is roughly a tube $l_{p,q}$  of radius $\delta$ inside $[0,1]^3$.

 To prove Theorem~\ref{tent}, it suffices to bound the number of distance quadruples.
 \begin{prop}\label{quadruple} Let $E$ be as in Theorem \ref{tent}.  Set $W = \delta^{-s/2}$, so that $E$ contains $\lesssim 1$ $\delta$-ball in each $1/W$-ball in $[0,1]^2$.   
 	Let $Q$ denote the collection of distance quadruples $(p_1, p_2, q_1, q_2)$ such that $$| dist(p_1, q_1)-dist(p_2, q_2)|< \delta,$$ for any $\delta$-balls   $p_1, p_2\in E_1, q_1, q_2\in E_2.$
 	Then $$\#Q\leq W^{8}\delta^{1-\epsilon}.$$
 \end{prop}
 Proposition~\ref{quadruple} implies Theorem~\ref{tent} because by Cauchy-Schwartz, 
 $$\#E_1 \cdot \# E_2 \leq \#\Delta(E_1, E_2)^{1/2} \cdot \#Q^{1/2}.$$ Now we turn to the proof of Proposition~\ref{quadruple}.
 
 \begin{proof}

  Let $\mathbb{T}$ denote the collection of $l_{p,q}$ and $l_{q,p}$ for all $p\in E_1$ and $q\in E_2$. A distance quadruple $(p_1, p_2, q_1, q_1)\in Q$ corresponds to the event that  $l_{p_1, q_2}$ and $l_{q_1, p_2}$ intersects at a $\delta$--ball(transversally). 
  
  To prove Proposition~\ref{quadruple}, it suffices to show that  
   the number of $r$--rich $\delta$--balls is bounded by 
  \begin{equation} \label{suffbound} |P_r(\mathbb{T})|\leq \delta^{1-s-\epsilon}\frac{|\mathbb{T}|^{3/2}}{r^2} \end{equation}
   because each $r$--rich $\delta$--ball corresponds to at most $r^2$ distance quadruples
   $$\#Q\leq \sum_{r \text{~~dyadic}} r^2 |P_r(\mathbb{T})|$$
   and we have  $W^2 = \delta^{-s}$.

  	When $r\lesssim \delta^{2-\epsilon/4} |\mathbb{T}| $, the estimate is true because $|P_r(\mathbb{T})|\lesssim  \delta^{-3}$, which is the maximum number of $\delta$--balls  in $[0,1]^3$.  When $r\gtrsim \delta^{2-\epsilon/4} |\mathbb{T}|$,  we want to apply Theorem \ref{main}.  Once we check that $\TT$ obeys the spacing hypotheses in Theorem \ref{main}, the theorem will give the bound \ref{suffbound}.  
  	
	To finish the proof, we check that  tubes in $\mathbb{T}$ have the good spacing property.  We can decompose the sphere $\mathbb{S}^2$ into union of caps $\tau$ of radius $1/W$. For each $\tau$, we can cover $[0,1]^3$ by finitely overlapping tubes of radius $1/W$ pointing on the direction in $\tau$.  Each $\delta$--tube in $\mathbb{T}$ corresponds to a unique pair of $\delta$--balls $(p,q)$. This is essentially the same reason as the one-by-one correspondence between the line $l_{xy}$ and the pair of  points  $(x,y)$. 
	Each $1/W$--tube in $[0,1]^3$ corresponds to a unique pair of $1/W$--squares $(Q_1, Q_2)$, $Q_i \subset [0,1]^2$.  Now the $\delta$-tube of $\TT$ corresponding to $(p,q)$ lies inside the $1/W$ tube corresponding to $(Q_1, Q_2)$ if and only if $p \in Q_1$ and $q \in Q_2$.  Since $E$ contains $\lesssim 1$ $\delta$-balls in any $1/W$-ball, each $1/W$-tube contains $\lesssim 1$ tube of $\TT$.  Moreover, $| \TT | \sim | E|^2 \sim \delta^{-2s} \sim W^4$.  So $\TT$ verifies the spacing hypotheses of Theorem \ref{main}.

\end{proof}

\end{document}